\documentclass[11pt,reqno]{amsart}

\usepackage{amsmath,amsthm,amsfonts,amssymb,latexsym,mathrsfs,color,hyperref}

\newtheorem{theorem}{Theorem}

\newtheorem{corollary}[theorem]{Corollary}
\newtheorem{proposition}[theorem]{Proposition}
\newtheorem{example}[theorem]{Example}

\newcommand{\pk}{{\rm pk\,}}
\newcommand{\lpk}{{\rm pk^{\emph{l}}\,}}

\newcommand{\man}{\mathfrak{A}_n}
\newcommand{\ma}{\mathfrak{A}_1}
\newcommand{\mb}{\mathfrak{A}_2}
\newcommand{\mc}{\mathfrak{A}_3}
\newcommand{\md}{\mathfrak{A}_4}
\newcommand{\msn}{\mathfrak{S}_n}

\newcommand{\lrf}[1]{\lfloor #1\rfloor}
\newcommand{\lrc}[1]{\lceil #1\rceil}
\newcommand{\as}{{\rm as\,}}
\newcommand{\Eulerian}[2]{\genfrac{<}{>}{0pt}{}{#1}{#2}}

\title{Enumeration of permutations by number of alternating runs}

\author{Shi-Mei Ma}
\address{School of Mathematics and Statistics,
Northeastern University at Qinhuangdao,
Hebei 066004,China}
\email{shimeima@yahoo.com.cn (S.-M. Ma)}
\thanks{This work was supported by NSFC (11126217) and the Fundamental Research Funds for the Central Universities (N100323013).}
\date{\today}

\begin{document}

\begin{abstract}
Let $R(n,k)$ denote the number of permutations of $\{1,2,\ldots,n\}$ with
$k$ alternating runs. We find a grammatical description of the numbers $R(n,k)$
and then present several convolution formulas involving the generating function for the numbers $R(n,k)$.
Moreover, we establish a connection between alternating runs and Andr\'e permutations.
\end{abstract}

\maketitle
\section{Introduction}
Let $\msn$ denote the symmetric group of all permutations of $[n]$, where $[n]=\{1,2,\ldots,n\}$.
Let $\pi=\pi(1)\pi(2)\cdots \pi(n)\in\msn$. We say that $\pi$ changes
direction at position $i$ if either $\pi({i-1})<\pi(i)>\pi(i+1)$, or
$\pi(i-1)>\pi(i)<\pi(i+1)$, where $i\in\{2,3,\ldots,n-1\}$. We say that $\pi$ has $k$ {\it alternating
runs} if there are $k-1$ indices $i$ such that $\pi$ changes
direction at these positions. Let $R(n,k)$ denote the number of
permutations in $\msn$ with $k$ alternating runs.
There is a large literature devoted to the numbers $R(n,k)$ (see~\cite[\textsf{A059427}]{Sloane}). The reader is referred to~\cite{BE00,CW08,Ma120,Ma122,Sta08} for recent results on this subject.

Andr\'e~\cite{Andre84} was the first to study
the alternating runs of permutations and he obtained the following recurrence relation
\begin{equation}\label{rnk-recurrence}
R(n,k)=kR(n-1,k)+2R(n-1,k-1)+(n-k)R(n-1,k-2)
\end{equation}
for $n,k\ge 1$, where $R(1,0)=1$ and $R(1,k)=0$ for $k\ge 1$.
For $n\geq 1$, we define
$R_n(x)=\sum_{k=1}^{n-1}R(n,k)x^k$.
Then by~\eqref{rnk-recurrence}, we obtain
\begin{equation}\label{Rnx-recurrence}
R_{n+2}(x)=x(nx+2)R_{n+1}(x)+x\left(1-x^2\right)R_{n+1}'(x),
\end{equation}
with initial value $R_1(x)=1$. The first few terms of $R_n(x)$'s are given as follows:
\begin{align*}
  R_2(x)& =2x, \\
  R_3(x)& =2x+4x^2, \\
  R_4(x)& =2x+12x^2+10x^3,\\
  R_5(x)& =2x+28x^2+58x^3+32x^4.
\end{align*}

The {\it Eulerian number} $\Eulerian{n}{k}$ enumerates the number of permutations in $\msn$
with $k$ descents (i.e., $1\leq i<n,\pi(i)>\pi(i+1)$).
Let
$A_n(x)=x\sum_{k=0}^{n-1}\Eulerian{n}{k}x^{k}$
be the {\it Eulerian polynomials}.
The polynomial $R_n(x)$ is closely related to $A_n(x)$:
\begin{equation*}\label{rnx-anx}
R_n(x)=\left(\dfrac{1+x}{2}\right)^{n-1}(1+w)^{n+1}A_n\left(\dfrac{1-w}{1+w}\right),\quad
w=\sqrt{\frac{1-x}{1+x}},
\end{equation*}
which was first established by David and
Barton~\cite[157-162]{DB62} and then stated more concisely by
Knuth~\cite[p.~605]{Knuth73}. In a series of papers~\cite{Carlitz78,Carlitz80,Carlitz81}, Carlitz studied the
generating functions for the numbers $R(n,k)$. In particular,
Carlitz~\cite{Carlitz78} proved that
\begin{equation}\label{CarlitzGF}
\sum_{n=0}^\infty \frac{z^n}{n!}\sum_{k=0}^nR(n+1,k)x^{n-k}=\frac{1-x}{1+x}\left(\frac{\sqrt{1-x^2}+\sin (z\sqrt{1-x^2})}{x- \cos (z\sqrt{1-x^2})}\right)^2.
\end{equation}
Recently, B\'ona and Ehrenborg~\cite[Lemma 2.3]{BE00} combinatorially proved that the polynomial $R_n(x)$ is divisible by $(1+x)^{\lrf{\frac{n}{2}}-1}$.

Let $\pi=\pi(1)\pi(2)\cdots \pi(n)\in\msn$. An {\it interior peak} in $\pi$ is an index $i\in\{2,3,\ldots,n-1\}$ such that $\pi(i-1)<\pi(i)>\pi(i+1)$.
Let $\pk(\pi)$ denote {\it the number of
interior peaks} in $\pi$. An {\it left peak} in $\pi$ is an index $i\in[n-1]$ such that $\pi(i-1)<\pi(i)>\pi(i+1)$, where we take $\pi(0)=0$.
Let $\lpk(\pi)$ denote {\it the number of
left peaks} in $\pi$. For example, the permutation $\pi=21435$ has $\pk(\pi)=1$ and $\lpk(\pi)=2$.
Let $W({n,k})$ denote the number of permutations in $\msn$ with $k$ interior peaks, and let $\widetilde{W}({n,k})$ denote the number of permutations in $\msn$ with $k$ left peaks. The numbers $W({n,k})$ and $\widetilde{W}({n,k})$ arise often in combinatorics and other branches of mathematics (see~\cite{Petersen09,Ma121,Petersen07}).

For $n\geq 1$, we define
$$W_n(x)=\sum_{k\geq0}W({n,k})x^k\quad {\text and}\quad \widetilde{W}_n(x)=\sum_{k\geq0}\widetilde{W}({n,k})x^k.$$
It is well known that the polynomials $W_n(x)$ satisfy the recurrence relation
\begin{equation*}
W_{n+1}(x)=(nx-x+2)W_n(x)+2x(1-x)W'_n(x),
\end{equation*}
with initial values $W_1(x)=1,W_2(x)=2$ and $W_3(x)=4+2x$, and the polynomials $\widetilde{W}_n(x)$ satisfy the recurrence relation
\begin{equation*}
\widetilde{W}_{n+1}(x)(x)=(nx+1)\widetilde{W}_n(x)+2x(1-x)\widetilde{W}'_n(x)(x),
\end{equation*}
with initial values $\widetilde{W}_0(x)=\widetilde{W}_1(x)=1,\widetilde{W}_2(x)=1+x$
and $\widetilde{W}_3(x)=1+5x$ (see~\cite[\textsf{A008303,A008971}]{Sloane}).

For $n\geq 0$, we define $$P_n(\tan \theta)=\frac{d^n}{d\theta^n}\tan\theta.$$
It is clear that $P_0(x)=x$ and $P_{n+1}(x)=(1+x^2)P_n'(x)$ (see~\cite{Hoffman95}).
The polynomials $P_n(x)$ are known as derivative polynomials for tangent and have been extensively studied (see~\cite{Cvijovic10,Franssens07,Ma121,Ma122}).
The following results were respectively obtained in~\cite{Ma121} and~\cite{Ma122}.
\begin{theorem}
For $n\geq 2$, we have
\begin{equation*}
W_n(x)=\frac{1}{x}(x-1)^{\frac{n+1}{2}}P_n\left(\frac{1}{\sqrt{x-1}}\right),
\end{equation*}
and
\begin{equation*}
R_n(x)=\left(\frac{x+1}{2}\right)^{n-1}\left(\frac{x-1}{x+1}\right)^{\frac{1}{2}(n+1)}P_n\left(\sqrt{\frac{x+1}{x-1}}\right),
\end{equation*}
\end{theorem}
So the following corollary is immediate.
\begin{corollary}
For $n\geq 2$, we have
\begin{equation}\label{RnxWnx}
R_n(x)=\frac{x(1+x)^{n-2}}{2^{n-2}}W_n\left(\frac{2x}{1+x}\right).
\end{equation}
\end{corollary}

Let $\pi=\pi(1)\pi(2)\cdots \pi(n)\in\msn$.
An {\it alternating subsequence} of length $k$ is a subsequence $\pi({i_1})\cdots \pi({i_k})$ satisfying
$$\pi({i_1})>\pi({i_2})<\pi({i_3})>\cdots \pi({i_k}).$$
Let $\as(\pi)$ denote the length of the longest alternating subsequence of $\pi$.
For $n\geq 1$, we define
$$a_k(n)=\#\{\pi\in\msn:\as(\pi)=k\}.$$

Recently, Stanley~\cite{Sta08} initiated a study of the distribution of the length of the longest alternating subsequences of $\pi$. Put
$$F_k(x)=\sum_{n\geq 0}a_k(n)\frac{x^n}{n!}\quad {\text and}\quad A(x,t)=\sum_{k\geq 0}F_k(x)t^k.$$
Stanley~\cite[Theorem 2.3]{Sta08} showed that
$$A(x,t)=(1-t)\frac{1+\rho+2te^{\rho x}+(1-\rho)e^{2\rho x}}{1+\rho-t^2+(1-\rho-t^2)e^{2\rho x}},$$
where $\rho=\sqrt{1-t^2}.$

Let
$T_n(x)=\sum_{k=1}^na_k(n)x^k$. B\'ona~\cite[Section 1.3.2]{Bona12} found the following identity:
\begin{equation}\label{TnxRnx}
T_n(x)=\frac{1}{2}(1+x)R_n(x)\quad {\text for}\quad n\geq 2.
\end{equation}
Combining~\eqref{RnxWnx} and~\eqref{TnxRnx}, we get
\begin{equation}\label{Tnx-wnx}
T_n(x)=\frac{x(1+x)^{n-1}}{2^{n-1}}W_n\left(\frac{2x}{1+x}\right)\quad {\text for}\quad n\ge 1.
\end{equation}
Moreover, it follows from~\eqref{Rnx-recurrence} that the polynomials $T_n(x)$ satisfy the recurrence relation
\begin{equation}\label{Tnx-recurrence}
T_{n+1}(x)=x(nx+1)T_{n}(x)+x\left(1-x^2\right)T_{n}'(x).
\end{equation}
with initial values $T_0(x)=1$ and $T_1(x)=x$. Therefore,
\begin{equation}\label{ank-recurrence}
a_k(n)=ka_k(n-1)+a_{k-1}(n-1)+(n-k+1)a_{k-2}(n-1),
\end{equation}
with initial values $a_0(0)=a_1(1)=1$ and $a_k(0)=a_0(n)=0$ for $n,k\geq 1$.

It should be noted that
$a_k(n)$ also enumerates the number of permutations in $\msn$ that have $k$ up-down runs.
The {\it up-down runs} of a permutation $\pi$ are the alternating runs of $\pi$ endowed with a 0 in the front (see~\cite[\textsf{A186370}]{Sloane}). For example, the permutation $\pi=514632$ has 3 alternating runs and 4 up-down runs. The up-down runs of a permutation are closely related to the numbers of interior peaks and left peaks.
Clearly,
$a_1(n)=\widetilde{W}(n,0)=1$, corresponding to the identity permutation.
By analyzing the following two cases:
\begin{enumerate}
  \item [($c_1$)]a permutation $\pi$ starts in a descent (i.e., $\pi(1)>\pi(2)$);
  \item [($c_2$)]a permutation $\pi$ starts in an ascent (i.e., $\pi(1)<\pi(2)$),
\end{enumerate}
it is easy to verify that
$$\widetilde{W}(n,k)=a_{2k}(n)+a_{2k+1}(n)\quad {\text for}\quad 1\leq k\leq \lrf{{(n-1)}/{2}},$$
$${W}(n,k)=a_{2k+1}(n)+a_{2k+2}(n)\quad {\text for}\quad 0\leq k\leq \lrf{{(n-2)}/{2}}.$$
Moreover, applying the complement operation $\phi$ to $\pi$, i.e., $\phi(\pi(i))=\pi(n+1-i)$, it is evident that $$W(2k+1,k)=a_{2k+1}(2k+1), \quad  \widetilde{W}(2k,k)=a_{2k}(2k).$$

Let $C_n(x)=xW_n(x^2)+\widetilde{W}_n(x^2)$.
The polynomials $C_n(x)$ satisfy the recurrence relation
\begin{equation*}
C_{n+1}(x)=(1+nx^2)C_{n}(x)+x(1-x^2)C_n'(x) \quad {\text for}\quad n\geq 1,
\end{equation*}
with initial values $C_1(x)=1+x$, $C_2(x)=1+2x+x^2$ and $C_3(x)=1+4x+5x^2+2x^3$
(see~\cite[Section 2]{Ma121}).
Then we have the following result.
\begin{theorem}\label{thm-cnx}
For $n\geq 2$, we have
$$xC_{n}(x)=(1+x)T_n(x).$$
\end{theorem}
Thus, using~\eqref{Tnx-wnx}, we obtain
$$C_{n}(x)=\frac{(1+x)^{n}}{2^{n-1}}W_n\left(\frac{2x}{1+x}\right)\quad {\text for}\quad n\ge 1.$$

As pointed out by Canfield and Wilf~\cite[Section 6]{CW08},
the generating function for the numbers $R(n,k)$ can be elusive.
In this paper we further explore the generating function for the numbers $R(n,k)$.
In Section~\ref{grammars-sec}, we present a description of the numbers $R(n,k)$ and $a_k(n)$ using the notion of context-free grammars. As applications, we obtain several convolution formulas involving the polynomials $R_n(x),T_n(x),W_n(x)$ and $\widetilde{W}_n(x)$. In Section~\ref{And-Permutations}, we establish a connection between alternating runs and Andr\'e permutations.
\section{Context-free grammars}\label{grammars-sec}
The grammatical method was introduced by Chen~\cite{Chen93}
in the study of exponential structures in combinatorics.
Let $A$ be an alphabet whose letters are regarded as independent commutative indeterminates.
Following Chen~\cite{Chen93}, a {\it context-free grammar} $G$ over $A$ is defined as a set
of substitution rules replacing a letter in $A$ by a formal function over $A$.
The formal derivative $D$ is a linear operator defined with respect to a context-free grammar $G$.
For any formal functions $u$ and $v$, we have
$$D(u+v)=D(u)+D(v),\quad D(uv)=D(u)v+uD(v) \quad and\quad D(f(u))=\frac{\partial f(u)}{\partial u}D(u),$$
where $f(x)$ is a analytic function.
Using Leibniz's formula, we have
\begin{equation}\label{Dnab-Leib}
D^n(uv)=\sum_{k=0}^n\binom{n}{k}D^k(u)D^{n-k}(v).
\end{equation}
For example, if $G=\{x\rightarrow xy, y\rightarrow y\}$, then $$D(x)=xy,D(y)=y,D^2(x)=x(y+y^2),D^3(x)=x(y+3y^2+y^3).$$

It is well know that many combinatorial objects permit a description using the notion of context-free grammars.
Recall that the {\it Stirling number of the second kind}
${n \brace k}$ is the number of ways to partition $[n]$ into $k$ blocks.
Chen~\cite[Eq. (4.8)]{Chen93} found that
if $G=\{x\rightarrow xy, y\rightarrow y\}$,
then
\begin{equation*}
D^n(x)=x\sum_{k=1}^n{n \brace k}y^k.
\end{equation*}
In~\cite{Schett76},
Schett considered the grammar
\begin{equation*}
G=\{x\rightarrow yz, y\rightarrow xz,z\rightarrow xy\},
\end{equation*}
and established a relationship between the expansion of $D^n(x)$ and the Jacobi elliptic functions.
Dumont~\cite{Dumont79} established the connections
between Schett's grammar and permutations.
In~\cite{Dumont96}, Dumont considered chains of general substitution rules on words.

Let us now recall two results on context-free grammars.
\begin{proposition}[{\cite[Section 2.1]{Dumont96}}]\label{Dumont}
If $G=\{x\rightarrow xy, y\rightarrow xy\}$,
then
\begin{equation*}
D^n(x)=x\sum_{k=0}^{n-1}\Eulerian{n}{k}x^{k}y^{n-k}\quad {\text for}\quad n\ge 1.
\end{equation*}
\end{proposition}
\begin{proposition}[\cite{Ma121}]\label{Ma20121}
If $G=\{y\rightarrow yz,z\rightarrow y^2\}$,
then
\begin{equation*}\label{derivapoly-1}
D^n(y)=\sum_{k=0}^{\lrf{{n}/{2}}}\widetilde{W}({n,k})y^{2k+1}z^{n-2k} \quad {\text for}\quad n\ge 0,
\end{equation*}
\begin{equation*}\label{derivapoly-2}
D^n(z)=\sum_{k=0}^{\lrf{({n-1})/{2}}}W({n,k})y^{2k+2}z^{n-2k-1}\quad {\text for}\quad n\ge 1.
\end{equation*}
\end{proposition}
As a conjunction of Proposition~\ref{Dumont} and Proposition~\ref{Ma20121}, it is natural to consider the context-free grammar
\begin{equation}\label{grammar-dis}
G=\{x\rightarrow xy, y\rightarrow yz,z\rightarrow y^2\}.
\end{equation}
Thus $D(x)=xy,D(y)=yz$ and $D(z)=y^2$.
In the following discussion we will consider the grammar~\eqref{grammar-dis}. For convenience, we will always assume that $n\geq 1$.
The main result of this paper is the following.
\begin{theorem}\label{mthm1}
If $G=\{x\rightarrow xy, y\rightarrow yz,z\rightarrow y^2\}$,
then
\begin{equation}\label{derivapoly-3}
D^{n}(x^2)=x^2\sum_{k=1}^{n}R(n+1,k)y^kz^{n-k},
\end{equation}
\begin{equation}\label{derivapoly-4}
D^n(x)=x\sum_{k=1}^na_k(n)y^kz^{n-k}\quad {\text for}\quad n\ge 1.
\end{equation}
\end{theorem}
\begin{proof}
Note that $D(x^2)=2x^2y$ and $D^2(x^2)=2x^2yz+4x^2y^2$.
For $n\geq 1$,
we define
\begin{equation}\label{Dnx-def}
D^n(x^2)=x^2\sum_{k=1}^nM(n+1,k)y^kz^{n-k}
\end{equation}
Then $M(2,1)=R(2,1)=2,M(3,1)=R(3,1)=2$ and $M(3,2)=R(3,2)=4$.
From~\eqref{Dnx-def}, we get
\begin{align*}
  D^{n+1}(x^2)& =D(D^n(x^2)) \\
              & =x^2\sum_{k=1}^nkM(n+1,k)y^kz^{n-k+1}+2x^2\sum_{k=1}^nM(n+1,k)y^{k+1}z^{n-k} \\
              & +x^2\sum_{k=1}^n(n-k)M(n+1,k)y^{k+2}z^{n-k-1}.
\end{align*}
Thus
$$M(n+2,k)=kM(n+1,k)+2M(n+1,k-1)+(n-k+2)M(n+1,k-2).$$
Comparing with~\eqref{rnk-recurrence}, we see that
the coefficients $M(n,k)$ satisfy the same recurrence relation and initial conditions as $R(n,k)$, so they agree.
Using~\eqref{ank-recurrence},
the formula~\eqref{derivapoly-4} can be proved in a similar way and we omit the proof for brevity.
\end{proof}

For the grammar~\eqref{grammar-dis}, we have
$$D(xy)=D(xz)=xyz+xy^2.$$
Hence
$D^n(xy)=D^n(xz)$ for $n\geq 1$. Moreover, $D^{n+1}(x^2)=D^n(2x^2y)$.
Using~\eqref{Dnab-Leib}, we obtain the following identities:
$$D^n(x^2)=\sum_{k=0}^n\binom{n}{k}D^k(x)D^{n-k}(x),$$
$$D^n(2x^2y)=2D^n(x^2y)=2\sum_{k=0}^n\binom{n}{k}D^k(x)D^{n-k}(xy)=2\sum_{k=0}^n\binom{n}{k}D^{k}(x)D^{n-k+1}(x),$$
$$D^n(2x^2y)=2\sum_{k=0}^n\binom{n}{k}D^k(x^2)D^{n-k}(y),$$
$$D^{n+1}(x)=D^n(xy)=\sum_{k=0}^n\binom{n}{k}D^k(x)D^{n-k}(y),$$
$$D^n(xz)=\sum_{k=0}^n\binom{n}{k}D^k(x)D^{n-k}(z).$$
Thus we can immediately use Proposition~\ref{Ma20121} and Theorem~\ref{mthm1} to get several convolution formulas.
\begin{corollary}\label{cor6}
For $n\geq 1$, we have
\begin{equation}\label{convolution}
R_{n+1}(x)=\sum_{k=0}^n\binom{n}{k}T_k(x)T_{n-k}(x),
\end{equation}
$$R_{n+2}(x)=2\sum_{k=0}^n\binom{n}{k}T_k(x)T_{n-k+1}(x),$$
$$R_{n+2}(x)=2x\sum_{k=0}^n\binom{n}{k}R_{k+1}(x)\widetilde{W}_{n-k}(x^2),$$
$$T_{n+1}(x)=x\sum_{k=0}^n\binom{n}{k}T_k(x)\widetilde{W}_{n-k}(x^2),$$
$$T_{n+1}(x)=T_n(x)+x^2\sum_{k=0}^{n-1}\binom{n}{k}T_k(x){W}_{n-k}(x^2).$$
\end{corollary}

From Corollary~\ref{cor6}, we see that there is a close relationship between $R_n(x),T_n(x),W_n(x)$ and $\widetilde{W}_n(x)$.
In particular, it follows from~\eqref{convolution} that
\begin{equation}\label{convolution2}
\sum_{n=0}^\infty \frac{z^n}{n!}\sum_{k=0}^nR(n+1,k)x^{n-k}=\left(\sum_{n=0}^\infty \frac{z^n}{n!}\sum_{k=0}^na_k(n)x^{n-k}\right)^2.
\end{equation}
Combining~\eqref{CarlitzGF} and~\eqref{convolution2}, it is easy to verify that
\begin{equation*}
\sum_{n=0}^\infty \frac{z^n}{n!}\sum_{k=0}^na_k(n)x^{n-k}=-\sqrt{\frac{1-x}{1+x}}\left(\frac{\sqrt{1-x^2}+\sin (z\sqrt{1-x^2})}{x- \cos (z\sqrt{1-x^2})}\right).
\end{equation*}

We end this section by giving another characterization of the numbers $R(n,k)$ and
the proof follows along the same lines as the proof of~\eqref{derivapoly-3}. 
\begin{theorem}
If $G=\{x\rightarrow 2xy, y\rightarrow yz,z\rightarrow y^2\}$,
then
\begin{equation*}
D^{n}(x)=x\sum_{k=1}^{n}R(n+1,k)y^kz^{n-k}\quad {\text for}\quad n\ge 1.
\end{equation*}
\end{theorem}
\section{Relationship to Andr\'e permutations}\label{And-Permutations}
A permutation $\pi=\pi(1)\pi(2)\cdots \pi(n)\in\msn$ is called an {\it Andr\'e permutation}
if it satisfies the following conditions (see~\cite{Hetyei96,Purtill93}):
\begin{enumerate}
 \item [({$c_1$})] $\pi$ has no double descents, i.e., there is no $i\in \{2,3,\ldots,n-1\}$ such that $\pi(i-1)>\pi(i)>\pi(i+1)$;
  \item [({$c_2$})]For all $j,j'\in \{2,3,\ldots,n\}$ which satisfy $j<j'$, if $\pi(j-1)=\max \{\pi(j-1),\pi(j),\pi(j'-1),\pi(j')\}$ and $\pi(j')=\min \{\pi(j-1),\pi(j),\pi(j'-1),\pi(j')\}$, then
      there is a $j''$ such that  $j<j''<j'$ and $\pi(j'')<\pi(j')$.
\end{enumerate}
The Andr\'e permutations are a variant of simsun permutations and are closely related to the enumeration of the monomials of the cd-index of $\msn$.
An {\it augmented Andr\'e permutation} is an Andr\'e permutation with $\pi(n)=n$.
Denote by $\man$ the augmented Andr\'e permutations in $\msn$.
\begin{example}
$$\ma=\{1\},\quad \mb=\{12\},\quad \mc=\{123,213\},$$
$$\md=\{1234,1324,2134,2314,3124\}.$$
\end{example}
Let $d(n,k)$ denote the number of
permutations in $\man$ with $k-1$ left peaks (see~\cite[Section 3]{Chow11}). The number $d(n,k)$
also counts the number of increasing 0-1-2 trees on $[n]$ with $k$ leaves (see~\cite[\textsf{A094503}]{Sloane}).
Let $D_n(x)=\sum_{k\geq 1}d(n,k)x^k$.
It is well known that the polynomials $D_n(x)$ satisfy the recurrence relation
\begin{equation}
D_n(x)=nxD_{n-1}(x)+x(1-2x)D_{n-1}'(x),
\end{equation}
for $n\geq 1$, with initial value $D_0(x)=1$ (see~\cite[(7.1)]{Foata01}). The first few terms of $D_n(x)$'s are given as follows:
\begin{align*}
  D_1(x)& =x, \\
  D_2(x)& =x, \\
  D_3(x)& =x+x^2,\\
  D_4(x)& =x+4x^2,\\
  D_5(x)& =x+11x^2+4x^3.
\end{align*}

The Eulerian polynomial $A_n(x)$ admits several expansions in terms of different polynomial bases.
One representative example is the classical {\it Frobenius formula} (see~\cite{Chow08}):
\begin{equation}\label{Frobenius}
A_n(x)=\sum_{i=1}^ni!{n \brace i}x^i(1-x)^{n-i}.
\end{equation}

Motivated by the expansion~\eqref{Frobenius}, we find the following result.
\begin{theorem}
For $n\geq 2$, we have
\begin{equation*}
D_n(x)=\frac{1}{2}\sum_{k=1}^{n-1}R(n,k)x^k(1-x)^{n-1-k}.
\end{equation*}
Equivalently,
\begin{equation}\label{RnxDnx}
R_n(x)=2(1+x)^{n-1}D_n\left(\frac{x}{1+x}\right).
\end{equation}
\end{theorem}
\begin{proof}
Let $F_n(x)=\frac{1}{2}\sum_{k=1}^{n-1}R(n,k)x^k(1-x)^{n-1-k}$ for $n\geq 2$. Clearly, $F_2(x)=x,F_3(x)=x+x^2$ and $F_4(x)=x+4x^2$. Set $F_0(x)=1$ and $F_1(x)=x$.
Note that $$R_n(x)=2(1+x)^{n-1}F_n\left(\frac{x}{1+x}\right)\quad {\text for}\quad n\ge 2.$$
Using~\eqref{Rnx-recurrence}, we get
$$F_n\left(\frac{x}{1+x}\right)=\frac{nx}{1+x}F_{n-1}\left(\frac{x}{1+x}\right)+
\frac{x(1-x)}{(1+x)^2}F_{n-1}'\left(\frac{x}{1+x}\right).$$
Thus $F_n(x)=nxF_{n-1}(x)+x(1-2x)F_{n-1}'(x)$. Hence $F_n(x)$ satisfies the same recurrence relation and initial conditions as $D_n(x)$, so they agree.
\end{proof}

Note that $\deg D_n(x)=\lrc{\frac{n}{2}}$.
Combining~\eqref{TnxRnx} and~\eqref{RnxDnx}, the following corollary gives a combinatorial interpretation for the fact that the polynomial $T_n(x)$ is divisible by $(1+x)^{\lrf{\frac{n}{2}}}$ (see~\cite[Corollary 3.2]{Sta08}).
\begin{corollary}
For $n\geq 0$, we have
$$T_n(x)=(1+x)^{n}D_n\left(\frac{x}{1+x}\right).$$
\end{corollary}

\end{document}